\documentclass[11pt,psfig]{article}
\usepackage[latin1]{inputenc}
\usepackage{amsmath}
\usepackage{amsfonts}
\usepackage{amssymb}
\usepackage{geometry}
\usepackage{algorithmic}
\usepackage{algorithm}
\usepackage{tikz} 
\usepackage{graphicx}
\usepackage{pstricks,pst-grad}
\usepackage{hyperref}
\usepackage{pifont}
\usepackage{cite}
\usepackage{amsthm}
\usepackage{amsmath}
\usepackage{amsfonts}
\usepackage{epsfig}
\usepackage{balance}

\newtheorem{thm}{Theorem}

\newtheorem{prop}{Proposition}

\newtheorem{lem}{Lemma}

\newcommand{\LL}{\hat{L}}
\newcommand{\UU}{\hat{U}}
\newcommand{\ZZ}{\bar{Z}}

\begin{document}
\begin{center}

{\Huge \textbf{The funds market bank problem}}

\mbox{}\\[0pt]
\vspace{0.2cm} Elena Cristina Canepa\\[0pt]  
Department of Mathematical Methods and Models\\[0pt]  
University Politehnica of Bucharest\\[0pt]  
Bucharest, Romania\\[0pt] 
cristinacanepa@yahoo.com
\vspace{1cm}

\vspace{0.2cm} Traian A.~Pirvu\\[0pt]  
Department of Mathematics \& Statistics\\[0pt]
McMaster University \\[0pt]
1280 Main Street West \\[0pt]
Hamilton, ON, L8S 4K1\\[0pt]
tpirvu@math.mcmaster.ca \vspace{1cm}
\mbox{}\\[0pt]

\end{center}


\begin{abstract}
This paper considers the problem faced by a bank which trades in the funds market
so as to maintain the reserve requirements and minimize the costs of doing that.
We work in a stochastic paradigm and the reserve requirements are determined by the demand deposit process, modelled as a geometric Brownian motion. The discount rates
for the cumulative funds purchased and the cumulative funds sold are assumed to be different. The optimal strategy of the bank is explicitly found and it has the following
structure: when bank reserves lower to an exogenously threshold level the bank has to purchase funds; when bank reserves tops an endogenously threshold level the bank has to sell funds.  
\end{abstract}
\section{Introduction}

The one bank problem we analyse in this paper looks at the micro-economic level with one bank and the fund markets. The fund markets usually consists of the central banks as the
main liquidity providers and other banks. The one bank we consider is a price-taker in the funds market and can obtain sufficient funds from the funds market. In our model the
bank has to borrow and lend funds to meet the reserve requirements as imposed by bank regulators. The bank's objective is to find the optimal transactions so as to minimize the cost of borrowing and lending funds needed for the reserve requirements. The bank's cost in implementing a borrowing and lending funds strategy consists of transaction costs.

Our paper is an extension of our previous work \cite{CP1} to allow for the modelling
of the reserve requirements as a geometric Brownian motion which is an improvement since
it implies the non negativity of the reserve requirement. In our previous work for tractability reasons we assume that the reserve requirements followed a Brownian motion
with drift. As in \cite{CP1} we allowed for the borrowing funds and lending funds to be discounted at possibly different rates. Our extension is in several directions. 1) we
assumed a positive exogeneously specified threshold (dictated by banking regulations) 
which is the minimum fund requirement for the bank. 2) As in \cite{CP1} we found a threshold which when attained it is optimal for the bank to start selling funds; this 
threshold is characterized through an algebraic equation whose existence is much more
difficult to establish than the corresponding one in  \cite{CP1}. 3) As in \cite{CP1}
the bank net purchase amount is described by a double Skorokhod formula, but because of the positive exogeneously specified threshold we had to recourse to the double Skorokhod formula of \cite{KavitaBurdzy}.

We formulate and solve the bank's problem by providing the banks's optimal value function and the optimal strategy.

The paper is structured as follows. In Section 2 we present the model and the main assumptions. In Section 3 we give the problem formulation and present the objective of the paper. In Section 4 we present the bank's optimal borrowing and lending funds policies and the main result of the paper.  
 
  \section{The model}\label{Descr}
In this paper we consider the problem faced by a bank which has an exogenously given demand deposit (net of withdrawals) and continuously sells and buys funds as to lower or increase the excess reserves. This is the difference between deposits and required reserves. The bank is described by:
   \begin{enumerate} 
   \item A demand deposit process $(D_{t})_{t \geq 0}$.
   
   \item  A required reserve process $(R_{t})_{t \geq 0}$, where $R_{t}=qD_{t}$.
   
   \item An excess reserve process $X_{t}=(1-q){D}_{t},$
 \end{enumerate} 

where the rate $q\in[0,1]$ is endogenously given. We work in a stochastic paradigm,  
 and $(\Omega, \textit{F},P_{x})$ is a probability space rich enough to accommodate
a standard, one-dimensional, Brownian motion $B = (B_{t}, 0 \leq t \leq \infty).$
We consider $\textbf{F}=(\textit{F}_{t})_{t\geq0}$ to be the completion of the augmented filtration generated by $X$ (so that $(\textit{F}_{t})$ satisfies the usual conditions).

The demand deposits $X = (X_{t}, 0 \leq t \leq \infty)$ are assumed to fluctuate over time as following a Geometric Brownian motion \begin{equation}\label{X}
dX_{t}= \mu X_{t} dt + \sigma X_{t} dB_{t}.
\end{equation}

so that 
$$X_{t}=X_{0}e^{{\left( \mu-  \frac{\sigma^2}{2}  \right)} t + {\sigma} B_{t}},$$ where $X_{0}>0, \sigma>0.$

Let us remark here that the process $X$ is positive and this is an improvement of our
model \cite{CP1} where for tractability reasons we assume $X$ to be a Brownian motion 
with drift.

We conclude this section noting that the bank observes nothing except the sample path of $X.$

 \subsection{Policies}

In the following we formally define bank's policies, i.e. amount of funds bought or sold in the funds market. 

\newtheorem{mydef}{Definition}[section]
 
\begin{mydef} \label{def:policy}
 A policy is defined as a pair of processes  $L$ and $U$ such that
\begin{equation}
\label{adaptLU}
 L, U \qquad \text{are}\qquad\textbf{F}-\text{adapted, right-continuous, increasing and positive.}
\end{equation}
In the context of the  funds market, $L_{t}$ and $U_{t}$ are the cumulative funds purchases and funds sales (from the central bank) that the bank undertakes up to time $t$, in order to satisfy the reserve requirements and to maximize its profit. Let us take $\lambda$ and $\\hat{lambda},$ $\lambda\geq \hat{\lambda}$ be interest rates at which the bank lends and borrows funds. A controlled process associated to the policy ($L,U$) is a process $Z = X+L-U$.
Using formula (\ref{X}) for $X$, we obtain the decomposition of $Z$ into its continuous part and its finite variation part:
\begin{equation}\label{Z}
dZ_{t}= \mu X_t dt + \sigma X_t dB_{t} + dL_{t} - dU_{t}.
\end{equation} 
In our model $Z_{t}$ is the amount of excess funds in the bank's reserve account at time $t.$ According to regulatory policies the bank should keep the amount of excess funds in the bank's reserve account above an exogenous level $a.$ The policy ($L,U$) is said to be feasible if
\begin{equation}
L_{0-}=U_{0-}=0,
\end{equation}

\begin{equation}
P_{x}\left\{ Z_{t} \geq a, \forall t \right \}=1, \forall x\geq 0,
\label{Z0}
\end{equation}

\begin{equation}
E_{x}\left[ \int_{0}^{\infty}e^{-\lambda_i t}dL \right] < \infty, \forall x\geq 0,\,\, i=1,2,
\label{Lintegrab}
\end{equation}

and
\begin{equation}
E_{x}\left[ \int_{0}^{\infty}e^{-\lambda_1 t}dU \right] < \infty, \forall x\geq 0.
\label{Uintegrab}
\end{equation}
We denote by ${\textit{S}}(x)$ the set of all feasible policies associated with the continuous process $X$ that starts at $x$.
\end{mydef}

 \subsection{Transaction Costs}

We assume that the bank can continuously sell and buy funds, thus lowering or increasing its excess reserve account. Following \cite{ChenMazumdar} we consider three types of transaction costs:
\begin{enumerate}
 \item A proportional transaction cost $\alpha$ of buying funds. 
 
 \item A proportional transaction cost $\beta$ of selling  funds.
 
 \item A continuous holding cost, incurred at the rate $h.$
\end{enumerate}

\section{Bank Optimization Problem}
 
 \subsection{The Cost Function}
 
 \begin{mydef}\label{d:k}
  The \textsl{cost function} associated to the feasible policy $(L,U)$ is 
  \begin{equation}
k_{L,U}(x)\equiv E_{x}\left[ \int_{0}^{\infty}[e^{-\lambda t}(hZ_{t}dt+\beta dU) + (n e^{-\lambda t} +(1-n) e^{-{ \hat{\lambda}} t})\alpha dL]        \right],\qquad  x\geq0,
\label{cost}
\end{equation}
with $n\in[0,1].$
\end{mydef}

 In our model, following \cite{CP1}, the cumulative funds purchases and funds sales are discounted at different rates $\hat{\lambda}\leq\lambda.$ If $n=1$ then the discounting occur at the same rate $\lambda.$

 \subsection{Banks's Objective}

    The bank's reserve management and profit-making problem is to find the optimal strategy $(\hat{L},\hat{U})$ which minimizes the cost. 

\begin{mydef}
 The control $(\hat{L},\hat{U})$ is said to be \textsl{optimal} if $k_{\hat{L},\hat{U}}(x)$ is minimal among the cost functions $k_{L,U}(x)$ associated with feasible policies $(L,U)$, for each fixed $x\geq 0$.
\end{mydef}
 
The problem cost minimization can be translated to gain maximization. The gain function is easier to work with as it turns out to have particular characteristics, when the policy is of a barrier type. We present the relation between the cost function and the gain function obtained by \cite{Harrison}.

 \subsection{The Gain Function}
 \begin{mydef}
 The gain function is defined by
  \begin {equation}
  v_{L,U}(x) \equiv E_{x}\left\{\int_{0}^{\infty} e^{-\lambda t}(rdU- cdL) \right\} - E_{x}\left\{\int_{0}^{\infty} e^{-\hat{\lambda} t}(1-n) \alpha dL  \right\},\qquad x\geq 0,
    \label{v}
  \end{equation}
  where $r \equiv h/\lambda-\beta,$ and $c \equiv h/\lambda+ n \alpha$. 
  \end{mydef}
  Then extending the arguments from \cite{Harrison} one gets the following Lemma.
  
  \begin{lem}\label{11}
   The relation between the cost function and the gain function is
      \begin{equation}\label{kv}
  k_{L,U}(x)= hx/\lambda + h\mu/\lambda^{2}- v_{L,U}(x), x\geq0.
  \end{equation}
 \end{lem}
 
\section{The Optimal Policy}
 
 It turns out that banks's optimal policies are of barrier type and we formally introduce them below.
 \subsection{The Barrier Policies}
 
Let $b>a>0$ be a real fixed number. We consider that $X_{0}=x\in [a,b]$. If $X_{0}>b$, then we allow a jump at $0$ for $U$: $U_{0}=X_{0}-b$.

   \begin{mydef}\label{Sc}
     The barrier policies are the set of policies $(L,U)\in \textit{S}(x)$ that satisfy:
     \begin{enumerate}
      \item  $(L,U)$ continuous on $(0,\infty)$, increasing, $L_{0-}=U_{0-}=0$,
      
       \item $Z_{t}\equiv X_{t}+L_{t}-U_{t}\geq 0, \forall t\geq 0$,
        
      \item $ \int_{0}^{t}I_{Z_{t}>a}dL_{t}=0, \int_{0}^{t}I_{Z_{t}<b}dU_{t}=0.$ 
      \end{enumerate} 
   \end{mydef}
   
   Let us recall that the lower threshold $a$ is endogenously given (impose by bank regulation). The upper threshold $b$ is chosen by the bank, it is the amount beyond which the bank will start selling funds.  
   
    The Double Skorokhod Formula obtained in \cite{KavitaBurdzy} can be translated into a formula for the bank's  transaction amount $L-U$:

  \begin{equation} \label{XLU}
L_{t} - U_{t} = -[(X_{0}-b)^{+}\wedge \inf_{u\in [0,t]}(X_{u}-a)] \vee \sup_{s\in [0,t]}[(X_{s}-b)\wedge \inf_{u\in[s,t]}
(X_{u}-a)].
\end{equation}
 
  \subsection{The Main Result}
  
  The bank has to optimally chose threshold $b$ as to maximize its gain function. The 
  optimal selection of $b$ is explained below.
  
 We need to introduce some quantities at this point. Let $-\gamma_{1}, \gamma_{2}$ be the roots of $ \sigma^{2}\gamma^{2}/2 + (\mu-\sigma^{2}/2 )\gamma - \lambda =0,$ so 

\begin{equation}
\label{beta_}
\gamma_{1}\equiv \frac{\sqrt{(\mu-\sigma^{2}/2)^{2}+2\sigma^{2}\lambda}+(\mu-\sigma^{2}/2)}{\sigma^{2}}>0,
\end{equation}
\begin{equation}
\label{beta^}
\gamma_{2}\equiv \frac{\sqrt{(\mu-\sigma^{2}/2)^{2}+2\sigma^{2}\lambda)}-(\mu-\sigma^{2}/2)}{\sigma^{2}}>0, 
\end{equation}

The constant $\bar{\gamma}_{2}$ is defined by

\begin{equation}
\label{beta^1}
\bar{\gamma}_{2}\equiv \frac{\sqrt{(\mu-\sigma^{2}/2)^{2}+2\sigma^{2}\bar{\lambda})}-(\mu-\sigma^{2}/2)}{\sigma^{2}}>0,
\end{equation}

Let the function $g$ be defined as
\begin{equation}\label{gG}
g(x)\equiv \gamma_{1}x^{\gamma_{2}} + \gamma_{2}x^{-\gamma_{1}}. 
\end{equation}  
 
 and
 
  \begin{equation}
  \label{vGb}
  v_{b}(x) =
\left\{
	\begin{array}{ll}
		 \frac{r}{g'(b/a)}g(x/a)+ \frac{c }{g'(a/b)}g(x/b) & \mbox{if } a\leq x\leq b \\
		 v_{b}(b)+ (x-b)r & \mbox{if }  x>b.
	\end{array}
\right.
\end{equation}

According to \cite{CanepaC} there exists a unique $b>0$ such that 
\begin{equation}\label{b}
\frac{g(1)}{g(a/b)}=\frac{r}{c},
\end{equation} 
which we denote by $b^{*}$.

Let us define 
 \begin{equation}
  \label{v1}
  v_{1} (x)=v_{b^{*}}(x),
\end{equation}

and
 
 \begin{equation}\label{v2}
 v_{2}(x) \equiv  - \frac{(1-n) \alpha   }{  \gamma_{2} a^{ \gamma_{2}-1  } } x^{\gamma_{2}}. 
\end{equation}

 \begin{prop}
  \label{BarrierValue}
  The barrier policy $(\hat{L}, \hat{U})$ associated with $b$ of \eqref{b} is admissible, i.e., $(\hat{L}, \hat{U})\in \textit{S}(x).$ Moreover
   \begin {equation}
  v_{1}(x) = E_{x}\left\{\int_{0}^{\infty} e^{-\lambda_1 t}(rd\hat{U}- cd \hat{L}) \right\},\qquad x\geq 0,
    \label{vv1}
  \end{equation}
   and
   \begin {equation}
  v_{2}(x) = - E_{x}\left\{\int_{0}^{\infty} e^{-\lambda_2 t}(1-n) \alpha d \hat{L}  \right\},\qquad x\geq 0.
    \label{vv2}
    \end{equation}
   Therefore
   \begin {equation}
  v_{\hat{L},\hat{U}}(x)= v_{1}(x) + v_{2}(x).
    \label{v3}
  \end{equation}
  
  \end{prop}

  \begin{proof}
  
If $u:R\rightarrow R$ is a function of class $C^{2}$ (i.e. twice continuously differentiable), then denote by $\Gamma$ the generator of the continuous diffusion process $X$ in (\ref{X}):
\begin{equation}\label{Gamma}
\Gamma u(x) = \mu x u'(x)+\frac{\sigma^{2}}{2} x^2 u''(x).
\end{equation}
 Since $Z=X+\LL-\UU$ then
\begin{equation}\label{R2}
 d v_{1}(Z_{t})=\sigma v_{1}'(Z)dB_{t} + [\Gamma v_{1}(Z)dt + v_{1}'(a)d\LL - v_{1}'(b)d\UU].
 \end{equation}
 Indeed by Ito's Lemma combined with the fact that $\LL$ increases only when $Z=a$, whereas $\UU$ increases only when $Z=b$ yields:
 \begin{align*}
 dv_{1}(Z_{t})&= v_{1}'(Z_{t})dZ_{t}+ 1/2v_{1}''(Z_{t})(dZ_{t})^{2}\\
  &= v_{1}'(Z)(dX + d\LL- d\UU)+ \frac{1}{2}\sigma^{2}v_{1}''(Z)dt\\
  &=v_{1}'(Z)(\mu dt+\sigma dB_{t} + d\LL_{t}-d\UU_{t})+\frac{1}{2}\sigma^{2}v_{1}''(Z)dt\\
  &= \sigma v_{1}'(Z)dB_{t}+ \Gamma v_{1}(Z)dt + v_{1}'(a)d\LL-v_{1}'(b)d\UU.
 \end{align*}
 Since $Z$ is bounded a.s. then $e^{-\lambda s}v_{1}'(Z)$ is bounded a.s., thus the process
 $M_{t}\equiv \int_{0}^{t}e^{-\lambda s}v_{1}'(Z_{t})\sigma dB_{t}, t\geq 0$ is a martingale. Consequently
 $$E_{x}M_{t}=0.$$
 From the definition of $ v_{1}$ we infer that (see for details \cite{CanepaC})
 $$\Gamma v_1(z) = \lambda v_1(z),\, z\in[a,b],\quad  v_{1}'(a)=c,\quad   v_{1}'(b)=r.   $$
  Applying integration by parts leads to:
 \begin{align*}
 e^{-\lambda t}v_{1}(Z_{t}) &= v_{1}(Z_{0}) + \int_{0}^{t}e^{-\lambda s}dv_{1}(Z) - \lambda \int_{0}^{t}e^{- \lambda s}v_{1}(Z)ds\\
 &=v_{1}(Z_{0}) + M_{t} +\int_{0}^{t}e^{- \lambda s}[\Gamma v_{1}(Z)ds+ v_{1}'(a)d\LL-v_{1}'(b)d\UU]- \lambda\int_{0}^{t}e^{- \lambda s} v_{1}(Z)ds \\
 &=v_{1}(Z_{0}) + M_{t} +\int_{0}^{t}e^{- \lambda s}[\Gamma v_{1}(Z)-\lambda v_{1}(Z)]ds - \int_{0}^{t}e^{- \lambda s}[v_{1}'(b)d\UU- v_{1}'(a)d\LL]\\
  &=v_{1}(Z_{0})+ M_{t}- \int_{0}^{t}e^{- \lambda s}[r d\UU-cd\LL].
  \end{align*}
  
  By taking expectation, then letting $t\rightarrow\infty,$ and using that $v_{1}$ is bounded leads to
  
  \begin{equation}\label{e0}
  v_{1}(x) = E_{x}\left\{\int_{0}^{\infty} e^{-\lambda t}(rd\hat{U}- cd \hat{L}) \right\}.
  \end{equation}
  
  Let $\ZZ=X+\LL$ then
\begin{equation}\label{RR2}
 d v_{2}(\ZZ_{t})=\sigma v_{2}'(\ZZ)dB_{t} + [\Gamma v_{2}(\ZZ)dt + v_{2}'(a)d\LL].
 \end{equation}
 Indeed by Ito's Lemma combined with the fact that $\LL$ increases only when $Z=a$ yields:
 \begin{align*}
 dv_{2}(\ZZ_{t})&= v_{2}'(\ZZ_{t})d\ZZ_{t}+ 1/2v_{2}''(\ZZ_{t})(d\ZZ_{t})^{2}\\
  &= v_{2}'(\ZZ)(dX + d\LL)+ \frac{1}{2}\sigma^{2}v_{2}''(\ZZ)dt\\
  &=v_{2}'(\ZZ)(\mu dt+\sigma dB_{t} + d\LL_{t})+\frac{1}{2}\sigma^{2}v_{2}''(\ZZ)dt\\
  &= \sigma v_{2}'(\ZZ)dB_{t}+ \Gamma v_{2}(\ZZ)dt + v_{2}'(a)d\LL.
 \end{align*}
 Since $\ZZ$ is positive a.s. then $e^{-\bar{\lambda} s}v_{2}'(\ZZ)$ is bounded a.s., thus the process
 $N_{t}\equiv \int_{0}^{t}e^{- \bar{\lambda} s}v_{2}'(\ZZ_{t})\sigma dB_{t}, t\geq 0$ is a martingale. Consequently
 $$E_{x}N_{t}=0.$$
 From the definition of $ v_{2}$ we infer that
 $$\Gamma v_2 = \bar{\lambda} v_2,\quad  v_{2}'(a)= -(1-n)\alpha.$$
  Applying integration by parts leads to:
 \begin{align*}
 e^{-\lambda_2 t} v_{2}(\ZZ_{t}) &= v_{2}(\ZZ_{0}) + \int_{0}^{t}e^{-\bar{\lambda} s}dv_{2}(\ZZ) - \bar{\lambda} \int_{0}^{t}e^{- \bar{\lambda} s}v_{2}(\ZZ)ds\\
 &=v_{2}(\ZZ_{0}) + N_{t} +\int_{0}^{t}e^{- \bar{\lambda} s}[\Gamma v_{2}(\ZZ)ds+ v_{2}'(a)d\LL]- \bar{\lambda}\int_{0}^{t}e^{- \bar{\lambda} s} v_{2}(\ZZ)ds \\
 &=v_{2}(\ZZ_{0}) + N_{t} +\int_{0}^{t}e^{- \bar{\lambda} s}[\Gamma v_{2}(\ZZ)-\bar{\lambda} v_{2}(\ZZ)]ds - \int_{0}^{t}e^{- \bar{\lambda} s}[- v_{2}'(a)d\LL]\\
  &=v_{2}(Z_{0})+ N_{t}- \int_{0}^{t}e^{- \lambda_2 s}[(n-1)\alpha d\LL].
  \end{align*}
  
  By taking expectation, then letting $t\rightarrow\infty,$ and using that $v_{2}$ is bounded on $[a,b]$ leads to
  
  \begin{equation}\label{e1}
   v_{2}(x) = - E_{x}\left\{\int_{0}^{\infty} e^{-\bar{\lambda} t}(1-n) \alpha d \hat{L}  \right\}.
\end{equation}

The feasibility of the policy $(\LL, \UU),$ can be established as in \cite{CP1}.

\end{proof}

 The following is the main result of our paper.
 \begin{thm}\label{main}
  The barrier policy $(\hat{L}, \hat{U})$ associated with $b$ of \eqref{b} is optimal, i.e., for every $({L}, {U})\in \textit{S}(x),$
   \begin {equation}
  v_{L,U}(x)\leq v_{  \hat{L}, \hat{U}}(x).
    \label{v4}
   \end{equation}
 \end{thm}

\begin{proof}
The proof follows the same arguments as in \cite{CP1} hence is omitted.
\end{proof}

\end{document}